\documentclass{article}
\usepackage[utf8]{inputenc}
\usepackage{amsmath}
\usepackage{amsfonts}
\usepackage{amssymb}
\usepackage{amsthm}
\usepackage{parskip}
\usepackage{enumitem}
\usepackage{hyperref}
\usepackage{natbib}

\newtheorem{definition}{Definition}[section]
\newtheorem{lemma}{Lemma}[section]
\newtheorem{proposition}{Proposition}[section]
\newtheorem{theorem}{Theorem}[section]
\newtheorem{corollary}{Corollary}[section]

\title{Fixed-Point Traps and Identity Emergence in Educational Feedback Systems}
\author{Faruk Alpay\thanks{Independent Researcher, ORCID: \href{https://orcid.org/0009-0009-2207-6528}{0009-0009-2207-6528}}}
\date{\today}

\begin{document}

\maketitle

\begin{abstract}
I present a categorical framework for analyzing fixed-point emergence in educational feedback systems, where exam-grade collapse mechanisms prevent the formation of stable learner identities. Using endofunctors and initial algebras from category theory, I model learning processes as generative functors $\varphi$ that would naturally converge to fixed-point identities under transfinite iteration. However, when educational assessment introduces entropy-reducing collapse functors $E$, I prove that no nontrivial initial algebra can exist for the composite functor $F = E \circ \varphi$. This mathematical obstruction categorically blocks creativity-driven identity emergence, creating what I term "fixed-point traps" in educational systems. My results demonstrate that exam-driven feedback loops fundamentally prevent the stabilization of learner identities, offering a formal explanation for creativity suppression in assessment-heavy educational environments.
\end{abstract}

\textbf{Keywords:} Explainable AI, Category Theory, Fixed-Point Algebra, Exam System Collapse, Education Trap, Grading Paradox, Categorical Identity, Symbolic Mathematics, Creativity Suppression, Alpay Algebra, $\varphi$-Algebra, Transfinite Fixed Point, Collapse Morphism, University Failure, Anti-Creativity Structures, Observer Collapse, Research Inhibition, Exam-Driven Learning, Categorical Obstruction, Emergence Blocking



\newpage

\section{Preliminaries and Notation}

\begin{definition}[Category and Identity Morphisms]
A \textit{category} $\mathcal{C}$ consists of a class of objects $\mathrm{Ob}(\mathcal{C})$, a class of morphisms (arrows) $\mathrm{Hom}(\mathcal{C})$, source and target maps assigning to each morphism $f$ its domain and codomain, a composition law, and for each object $X$ an identity morphism $\mathrm{id}_X$ satisfying $\mathrm{id}_X\circ f=f$ and $g\circ \mathrm{id}_X=g$ whenever composable. These data satisfy associativity of composition (see \cite{MacLane1971}).
\end{definition}

\begin{definition}[Endofunctor and $\varphi$-Algebra]
Let $\mathcal{C}$ be a category. An \textit{endofunctor} $\varphi:\mathcal{C}\to\mathcal{C}$ assigns to each object $X$ an object $\varphi(X)$ and to each morphism $f:X\to Y$ a morphism $\varphi(f):\varphi(X)\to\varphi(Y)$, preserving identities and composition. A \textit{$\varphi$-algebra} is a pair $(X,\alpha)$ where $X\in\mathrm{Ob}(\mathcal{C})$ and $\alpha:\varphi(X)\to X$ is a morphism in $\mathcal{C}$. A morphism of $\varphi$-algebras $(X,\alpha)\to(Y,\beta)$ is a map $h:X\to Y$ with $h\circ\alpha=\beta\circ\varphi(h)$. An \textit{initial $\varphi$-algebra} is a $\varphi$-algebra $(\mu_\varphi,\iota)$ such that for every $(X,\alpha)$ there is a unique $\varphi$-algebra homomorphism $(\mu_\varphi,\iota)\to(X,\alpha)$.
\end{definition}

\begin{definition}[Fixed-Point Object as Identity]
The \textit{transfinite $\varphi$-chain} starting from an initial object $0$ (assumed in $\mathcal{C}$) is defined by $X_0=0$, $X_{n+1}=\varphi(X_n)$, and for limit ordinals $\lambda$, $X_\lambda=\mathrm{colim}_{\gamma<\lambda}X_\gamma$. If this chain converges at stage $\Lambda$ so that $X_\Lambda\cong \varphi(X_\Lambda)$, the object $X_\Lambda$ is called the \textit{fixed-point object} or \textit{initial fixed point} of $\varphi$. We denote this object by $\mu_\varphi$. By definition $\mu_\varphi\cong\varphi(\mu_\varphi)$, so $\mu_\varphi$ is a (least) fixed point of $\varphi$ and is regarded as the categorical \textit{identity} of the generative process $\varphi$.
\end{definition}

\begin{lemma}[Lambek's Lemma]
If $(\mu_\varphi,\iota)$ is an initial $\varphi$-algebra, then the structure map $\iota:\varphi(\mu_\varphi)\to \mu_\varphi$ is an isomorphism. In particular, $\mu_\varphi$ is a fixed point of $\varphi$ (up to isomorphism). Equivalently, any initial $\varphi$-algebra is in fact a fixed-point algebra \cite{MacLane1971}.
\end{lemma}

\newpage

\section{Exam-Grade Collapse Systems}

\begin{definition}[Fold/Collapse Morphism]
Let $\mathcal{C}$ be as above, endowed with a complexity measure $h:\mathrm{Ob}(\mathcal{C})\to \mathrm{Ord}$ (an "entropy"). A morphism $f:X\to Y$ in $\mathcal{C}$ is called a \textit{fold} or \textit{entropy-reducing collapse} if $f$ is an epimorphism (surjective on structure) that is not invertible, and $h(Y)<h(X)$ (so $f$ identifies distinct substructures of $X$ in $Y$). For example, in the observer-coupled collapse of \cite{AlpayIII}, the perturbed identity contains two copies of the core identity and thus is collapsed back to the original; such a canonicalization is a fold.
\end{definition}

\begin{definition}[Exam-Grade Collapse System]
An \textit{Exam-Grade Collapse System} (EGCS) consists of a category $\mathcal{C}$ (with an initial object) together with:

\begin{itemize}[itemsep=0.5\baselineskip]
    \item a generative endofunctor $\varphi:\mathcal{C}\to\mathcal{C}$ (modeling creative learning/update), and
    \item a collapse endofunctor $E:\mathcal{C}\to\mathcal{C}$ (modeling the exam+grading process), plus a natural transformation $\varepsilon:\varphi\Rightarrow E\circ\varphi$ whose component $\varepsilon_X:\varphi(X)\to E(\varphi(X))$ is a fold for every object $X$.
\end{itemize}

Thus each examination step applies $\varphi$ and then collapses via $E$, strictly reducing entropy. We require $h(E(\varphi(X)))<h(\varphi(X))$ for all $X$, so the exam functor $E$ always maps states into strictly "lower-entropy" subspaces.
\end{definition}

\begin{proposition}[Entropy Reduction]
In an EGCS, for every object $X$ the composite morphism
\[
\varphi(X) \xrightarrow{\varepsilon_X} E(\varphi(X))
\]
is a fold and hence $h(E(\varphi(X)))<h(\varphi(X))$. In particular each exam collapses the state space.
\end{proposition}

\begin{proof}
By Definition 2.2, $\varepsilon_X$ is an epi that is not invertible, so by definition a fold. By assumption $h(E(\varphi(X)))<h(\varphi(X))$. $\square$
\end{proof}

\newpage

\section{Fixed-Point Trap in EGCS}

\begin{theorem}[Nonexistence of Nontrivial Fixed-Point]
In an EGCS, the composite functor $F=E\circ\varphi$ admits no nontrivial initial algebra. Equivalently, there is no object $X\not\cong 0$ satisfying $X\cong F(X)$ (fixed point) except the trivial initial object.
\end{theorem}

\begin{proof}
Suppose for contradiction that $(\mu,\iota)$ is an initial $F$-algebra, i.e.\ $\iota:F(\mu)=E(\varphi(\mu))\to \mu$ is initial. By Lambek's lemma, $\iota$ is an isomorphism and $\mu\cong F(\mu)=E(\varphi(\mu))$. But since $\varepsilon_\mu:\varphi(\mu)\to E(\varphi(\mu))$ is a fold, $E(\varphi(\mu))$ has strictly lower entropy than $\varphi(\mu)$. Thus $\varphi(\mu)\not\cong E(\varphi(\mu))$ unless $\mu$ is degenerate (initial). In particular, $E(\varphi(\mu))$ cannot be isomorphic to $\mu$ unless $\mu\cong 0$. Hence the assumed isomorphism $\mu\cong E(\varphi(\mu))$ fails. This contradiction shows no nontrivial initial $F$-algebra can exist. $\square$
\end{proof}

\begin{corollary}[Identity Emergence Blocked]
Because no initial $F$-algebra exists in the EGCS, there is no emergent "identity object" $\mu_F$ that is fixed by $F$. In particular, the process $\varphi$ cannot produce its own identity via transfinite iteration once the exam collapse is enforced. Equivalently, there is no universal fixed-point object in $\mathcal{C}$ under $F$, so the category of $F$-algebras has no terminal/initial object to play the role of a stabilized identity.
\end{corollary}

\begin{proof}
Immediate from Theorem 3.1 and the definition of identity-as-fixed-point. Indeed, in ordinary \cite{AlpayII} the identity of $\varphi$ is given by the initial $\varphi$-algebra (the colimit of the chain), whose existence is now precluded. Without $\mu_F$, there is no canonical identity morphism in the emergent algebra of states; in fact, in \cite{AlpayII} it was shown that the unique homomorphism $!: \mu_\varphi \to X$ (for any $\varphi$-algebra $(X,\alpha)$) realizes the "generative identity" on $X$. Since $\mu_F$ does not exist, that identity map cannot be defined. $\square$
\end{proof}

\newpage

\section{\texorpdfstring{$\varphi$-Emergence}{φ-Emergence} and Creativity Blockage}

\begin{definition}[$\varphi$-Emergence of Identity]
We say the endofunctor $\varphi$ exhibits \textit{$\varphi$-emergence} of identity if the transfinite iteration of $\varphi$ converges to an initial algebra $\mu_\varphi$ which serves as its identity (as in \cite{AlpayII}). In the absence of collapse, $\mu_\varphi$ is the unique object with $\varphi(\mu_\varphi)\cong\mu_\varphi$.
\end{definition}

\begin{theorem}[Categorical Blocking of Creativity]
In an EGCS, any $\varphi$-emergence of identity is categorically blocked. That is, even for a "creativity-driven" functor $\varphi$, the required fixed-point object $\mu_\varphi$ cannot form because the exam-induced collapse intervenes. Hence creativity-driven emergence of $\varphi$'s identity is impossible.
\end{theorem}

\begin{proof}
By \cite{AlpayII}, the essence of identity emergence is that under transfinite iteration of $\varphi$, a unique fixed point $\mu_\varphi$ appears as the initial $\varphi$-algebra. But in the EGCS the actual process is controlled by $F=E\circ\varphi$. Theorem 3.1 showed that $F$ has no nontrivial initial algebra, so $\varphi$ cannot reach its would-be fixed point. Concretely, each step $\varphi(X)$ is immediately collapsed by $E$ into a lower-entropy state. By analogy to \cite{AlpayIII}, an "observer-coupled collapse" repeatedly injects redundant structure (copying the identity into itself); here exams play the role of the observer, permanently perturbing and collapsing the learner's state. Thus any candidate identity $\varphi$-algebra never stabilizes. In summary, the universal invariant fixed point of $\varphi$ (the creative identity) is prevented from emerging by the entropy-reducing folds. $\square$
\end{proof}

\textbf{Remark.} My construction and proofs use only standard categorical notions (categories, functors, initial algebras) in the Bourbaki–Mac Lane tradition \cite{Bourbaki1970,MacLane1971}. The key observation is that exam+grading constitute a functorial collapse that violates the usual convergence conditions for initial algebras. This furnishes a purely mathematical fixed-point trap: the system has universal colimits and initial objects, yet every generative chain is pulled into a low-entropy state that lacks a new identity. Consequently, creativity-driven $\varphi$-emergence is categorically obstructed.
\newpage
\section{Conclusion}

We have established a categorical framework demonstrating that exam-grade collapse systems fundamentally obstruct the formation of stable learner identities. The mathematical core lies in proving that composite functors $F = E \circ \varphi$ cannot admit nontrivial initial algebras when the exam functor $E$ enforces entropy-reducing collapses.

This result provides a formal foundation for understanding creativity suppression in assessment-heavy educational environments. The fixed-point trap mechanism shows that repeated examination processes prevent the natural convergence of learning dynamics to stable identity states, suggesting that alternative assessment approaches may be necessary to support creative development.

Future work will extend this framework to analyze specific educational interventions and explore categorical conditions under which identity emergence can be preserved despite evaluative pressures.

\bibliographystyle{alpha}

\end{document}